\documentclass[11pt,letterpaper,reqno]{amsart}

\usepackage[T1]{fontenc}
\usepackage{amsmath,amssymb,amsthm,amsfonts,mathtools}
\usepackage{enumitem}
\usepackage{hyperref}
\hypersetup{colorlinks=true,linkcolor=blue,citecolor=blue,urlcolor=blue}

\newtheorem{theorem}{Theorem}[section]
\newtheorem{lemma}[theorem]{Lemma}
\newtheorem{proposition}[theorem]{Proposition}

\newtheorem{problem}[theorem]{Problem}

\theoremstyle{definition}
\newtheorem{definition}[theorem]{Definition}
\newtheorem{remark}[theorem]{Remark}
\numberwithin{equation}{section}

\newcommand{\C}{\mathbb{C}}
\newcommand{\OO}{\mathcal{O}}
\DeclareMathOperator{\Aut}{Aut}

\begin{document}

\title{A Solution to a Problem of Rubel on Two-Parameter Normal Families of Entire Functions}

\author[Y.~He]{Yixin He}
\address{School of Mathematical Sciences, Fudan University, Shanghai 200433, P.~R.~China}
\email{hyx717math@163.com}

\author[Q.~Tang]{Quanyu Tang}
\address{School of Mathematics and Statistics, Xi'an Jiaotong University, Xi'an 710049, P.~R.~China}
\email{tang\_quanyu@163.com}

\author[T.~Zhang]{Teng Zhang}
\address{School of Mathematics and Statistics, Xi'an Jiaotong University, Xi'an 710049, P.~R.~China}
\email{teng.zhang@stu.xjtu.edu.cn}



\begin{abstract}
We construct an entire function
\[
F(z,a,b)\in \OO(\C^3)
\]
such that the family
\[
\{F(\,\cdot\,,a,b):a,b\in\C\}
\]
of entire functions of $z$ is normal on $\C$, while $F$ does not factor through a single entire parameter. This solves a problem of L.~A.~Rubel concerning Liouville-type rigidity. In fact, our example satisfies the stronger condition
\[
F_bF_{a,z}-F_aF_{b,z}\neq 0
\qquad\text{on }\C^3.
\]
The geometric core of the construction is a Fatou--Bieberbach domain contained in the thin region
\[
\{(u,v)\in\C^2:|u-v^2|<1+|v|\}.
\]
We obtain this domain from the basin of attraction of an explicit polynomial automorphism of $\C^2$, together with the theorem of Rosay and Rudin on attracting basins.
\end{abstract}

\maketitle

\section{Introduction}\label{sec:intro}

Normal family theory furnishes one of the central compactness principles in complex analysis.
From Montel's foundational work \cite{Mon27} and Marty's spherical-derivative criterion \cite{Mar31} to Zalcman's rescaling lemma \cite{Zal98} and its refinements \cite{CFZ04,PZ00}, normality has repeatedly turned value-distribution constraints and differential relations into strong compactness conclusions.
On the complex plane, global normality is especially restrictive.
In the translation-invariant setting it is measured by bounded spherical derivative, and the theorem of Clunie and Hayman \cite{CH65} shows that an entire function with bounded spherical derivative has order at most one; see also Eremenko's higher-dimensional reformulation in terms of Brody curves \cite{Ere10}.
These results are representative of a broad principle: once normality is available on all of $\C$, severe global rigidity phenomena begin to appear.

A different but equally classical theme is factorization under composition.
Beginning with Ritt's work on polynomial decomposition \cite{Rit22} and Gross's factorization program for meromorphic functions \cite{Gro68,GY72}, one seeks conditions under which an analytic object that appears to depend on several layers of structure in fact factors through a lower-dimensional one.
In the entire-function setting this perspective was sharpened by Rubel and Yang, who showed that for nonconstant entire functions $A$ and $B$ of one variable, the two-variable entire function $A(z)+B(w)$ is prime under composition \cite{RY95}.
Eremenko and Rubel \cite{ER96} then developed an arithmetic of entire functions under composition, proving in particular that every family of nonconstant entire functions of one complex variable has a unique strong greatest common right factor, and more generally that any family of pairwise dependent entire functions in several variables admits a strong common right factor.
Their factorization theorem for degenerate maps states that if $F:\C^N\to\C^2$ has Jacobian rank at most one, then $F$ factors through $\C$ or $\C^\ast$, see \cite[Theorem~3.2]{ER96}.
Related manifestations of forced one-dimensionality appear in the theory of imprimitive parametrizations of analytic curves \cite{Ng01}, where algebraic or analytic dependence yields a transcendental common right factor.

A third surrounding theme concerns rigidity in holomorphic families of entire functions.
Classical work of Julia, Lelong, and Tsuji on exceptional parameter sets for equations $f(z,w)=0$ initiated a program for understanding how one-variable value-distribution phenomena propagate in parameter space; see the historical discussion in \cite{Ere06}.
In a modern formulation, Eremenko \cite{Ere06} proved that if each member of a holomorphic family of entire functions omits a value, then the exceptional value varies holomorphically off a discrete singular set, and in the finite-order regime the dependence is in fact analytic in the parameter.
Such questions are also closely related to holomorphic dynamics; see, for instance, the structural-stability viewpoint in the work of Eremenko and Lyubich \cite{EL92}.
The common theme is again that global function-theoretic restrictions can force surprisingly low-dimensional parameter behavior.

Against this background, L.~A.~Rubel raised the following Liouville-type rigidity question: instead of asking whether a global constraint forces an entire function to be constant, one asks whether global normality forces a holomorphic two-parameter family to become effectively one-parameter. The problem appears to have first been recorded in the survey of Brannan and Hayman \cite[p.~14, Problem~2.73]{BH89} in 1989, and was later incorporated into Hayman's influential problem collection \emph{Research Problems in Function Theory}, whose fiftieth-anniversary edition was published by Hayman and Lingham in 2019 \cite[p.~53, Problem~2.73]{HL19}. Moreover, no progress on this problem is recorded in \cite{Ere03}, and \cite[p.~54, Update~2.73]{HL19} reports that no progress had been communicated to the editors, indicating that the problem appears to have remained open for nearly four decades.

\begin{problem}[Rubel]\label{prob:HL273}
Does there exist an entire function
\[
F(z,a,b)\in \OO(\C^3)
\]
such that the family
\[
\mathcal{F}:=\{F(\,\cdot\,,a,b):a,b\in\C\}\subset \OO(\C)
\]
is normal on $\C$, but $F$ is not of the form
\begin{equation}\label{eq:factor}
F(z,a,b)=G\bigl(z,H(a,b)\bigr)
\end{equation}
for entire functions $G,H\in \OO(\C^2)$?
\end{problem}

The exclusion of \eqref{eq:factor} is meant to rule out merely apparent two-parameter dependence.
If \eqref{eq:factor} holds, then the family is in fact parametrized by the single entire quantity $H(a,b)$; in the language of the source, one wants the family to have ``two honest parameters.''
Otherwise, one could have, for instance,
\[
F(z,a,b)=z+5a^2+\sin b,
\]
which is really only a one-parameter family in disguise.
Moreover, \eqref{eq:factor} imposes an infinitesimal rank-one constraint.
Indeed,
\[
F_a=G_w(z,H(a,b))\,H_a,\qquad F_b=G_w(z,H(a,b))\,H_b,
\]
and hence
\begin{equation}\label{eq:integrability}
F_b\,F_{a,z}=F_a\,F_{b,z}.
\end{equation}
Accordingly, Problem~\ref{prob:HL273} asks whether global normality of a holomorphic two-parameter family already forces the parameter dependence to be effectively one-dimensional.

A positive answer to Problem~\ref{prob:HL273} would exhibit a genuinely two-dimensional entire deformation whose one-variable slices remain globally normal on the whole plane.
A negative answer to Problem~\ref{prob:HL273} would be even more striking: it would show that for entire families parametrized by $\C^2$, global normality alone compels the parameters to collapse to a one-dimensional analytic leaf.
From this point of view, the problem sits exactly at the interface between compactness in the sense of Montel and factorization in the sense of Ritt--Gross--Rubel.

The results of Eremenko and Rubel make the structural target especially clear.
Once one can show that the parameter dependence has rank at most one---for example, by deriving an identity of the type \eqref{eq:integrability} or an equivalent Jacobian degeneracy statement---the existing factorization machinery becomes available \cite{ER96}.
The genuine difficulty is therefore not factorization after degeneracy, but rather degeneracy from normality.
Put differently, Problem~\ref{prob:HL273} asks whether a compactness hypothesis on the family of slices $z\mapsto F(z,a,b)$ can force a first-order integrability condition in the parameter variables.

In this paper, our main result gives an affirmative answer to Problem~\ref{prob:HL273}. In fact, we prove something slightly stronger than the mere failure of \eqref{eq:factor}.

\begin{theorem}\label{thm:main}
There exists an entire function $F\in \OO(\C^3)$ such that
\begin{enumerate}[label=\textup{(\roman*)}]
\item the family
\[
\{F(\,\cdot\,,a,b):a,b\in\C\}
\]
is normal on $\C$;
\item
\[
F_bF_{a,z}-F_aF_{b,z}\neq 0
\qquad\text{on }\C^3.
\]
In particular, $F$ cannot be written in the form \eqref{eq:factor} with entire $G,H\in\OO(\C^2)$.
\end{enumerate}
\end{theorem}

Let us make explicit which parts of the paper are new and which are known. Theorem~\ref{thm:RR} below is the classical theorem of Rosay and Rudin on attracting basins, and Lemma~\ref{lem:normal-affine} is a short normality argument for a family of affine functions. The main new ingredient is Proposition~\ref{prop:fb-parabola}, where we construct a Fatou--Bieberbach domain inside
\[
S:=\{(u,v)\in\C^2:|u-v^2|<1+|v|\}.
\]
Once such a domain is available, Theorem~\ref{thm:main} follows by pulling back the affine family
\[
z\mapsto u+vz,\qquad (u,v)\in S,
\]
along a biholomorphism $\Phi=(U,V):\C^2\to\Omega$ and setting
\[
F(z,a,b):=U(a,b)+zV(a,b).
\]
The differential quantity in part~\textup{(ii)} then becomes the Jacobian determinant of $\Phi$, up to sign.

As for previous work, the literature we use falls into two separate strands. On the one hand, the problem itself belongs to the theory of normal families; for background on the rescaling philosophy surrounding such questions, especially in connection with Bloch's principle, we refer to Bergweiler's survey \cite{Ber06}. On the other hand, our construction relies on several complex variables: Rosay and Rudin provide the basin theorem used in Section~\ref{sec:fb}, while work such as Buzzard and Hubbard's shows that thin Fatou--Bieberbach domains can occur in surprisingly restrictive regions \cite{BH00}. Our contribution is to connect these two strands by producing a Fatou--Bieberbach domain in the specific coefficient region $S$, which is exactly what is needed to solve Problem~\ref{prob:HL273}.

\subsection{Paper organization}\label{subsec:organization}
Section~\ref{sec:fb} constructs a Fatou--Bieberbach domain inside the parabolic tube $S$. Section~\ref{sec:normal} proves that the corresponding affine family is normal on $\C$. In Section~\ref{sec:proof}, we combine these two ingredients to build the required entire function $F$, and we verify the differential obstruction that rules out factorization through a single entire parameter.

\section{A Fatou--Bieberbach domain in a parabolic tube}\label{sec:fb}
We begin by recalling the definition of a \emph{Fatou--Bieberbach domain}. Such domains abound in $\C^n$ for every $n>1$; see~\cite[Chapter~4]{For11} for a survey. For example, the basin of attraction of a holomorphic automorphism of $\C^n$ is either all of $\C^n$ or a Fatou--Bieberbach domain; see~\cite[Appendix]{RR88} and~\cite[Theorem~4.3.2]{For11}. 
\begin{definition}
A \emph{Fatou--Bieberbach domain} in $\C^2$ is a proper domain $\Omega\subsetneq\C^2$ which is biholomorphic to $\C^2$. 
\end{definition}

We will use the following theorem of Rosay and Rudin \cite[p.~49, ($*$)]{RR88}.

\begin{theorem}[Rosay--Rudin]\label{thm:RR}
Let $\Psi\in\Aut(\C^n)$ fix a point $p\in\C^n$, and suppose that every eigenvalue of $D\Psi(p)$ has modulus strictly less than $1$. Then the basin of attraction
\[
\{z\in\C^n:\Psi^m(z)\to p\text{ as }m\to\infty\}
\]
is biholomorphic to $\C^n$.
\end{theorem}

We now construct the coefficient domain needed later in the proof of Theorem~\ref{thm:main}.

\begin{proposition}\label{prop:fb-parabola}
Let
\[
S:=\{(u,v)\in\C^2:|u-v^2|<1+|v|\}.
\]
Then there exists a Fatou--Bieberbach domain $\Omega\subset S$.
\end{proposition}

\begin{proof}
Consider the polynomial automorphism of $\C^2$
\[
f(z_1,z_2):=\left(z_2,\frac{z_2^2-z_1}{2}\right),
\]
with inverse
\[
g(w_1,w_2)=f^{-1}(w_1,w_2)=\bigl(w_1^2-2w_2,w_1\bigr).
\]
Let
\[
D:=\{z\in\C^2:f^m(z)\to 0\text{ as }m\to\infty\}
\]
be the basin of attraction of the fixed point $0$.

To place $D$ inside a thin set, we first identify a small forward-invariant neighborhood of the origin.

\medskip
\noindent\textbf{Step 1: a local basin.}
Define a weighted norm on $\C^2$ by
\[
\|(z_1,z_2)\|_*:=\max\left\{\frac{3}{4}|z_1|,|z_2|\right\},
\]
and set
\[
B:=\left\{z\in\C^2:\|z\|_*<\frac16\right\}.
\]
If $r:=\|z\|_*<\frac16$, then $|z_2|\le r$ and $|z_1|\le \frac{4}{3}r$. Hence
\[
\|f(z)\|_*=
\max\left\{\frac34|z_2|,\left|\frac{z_2^2-z_1}{2}\right|\right\}
\le
\max\left\{\frac34 r,\frac{r^2+\frac43r}{2}\right\}
\le \frac34 r,
\]
because $r<\frac16$ implies $r^2+\frac43r\le \frac32 r$. Therefore
\[
\|f(z)\|_*\le \frac34\|z\|_*
\qquad\text{for every }z\in B.
\]
In particular, $f(B)\subset B$, and thus $f^m(z)\to 0$ for every $z\in B$.

This gives the usual exhaustion of the attracting basin:
\begin{equation}\label{eq:D_union}
D=\bigcup_{m\ge0} g^m(B).
\end{equation}
Indeed, if $z\in D$, then $f^m(z)\to 0$, so $f^m(z)\in B$ for some $m$, and hence $z=g^m(f^m(z))\in g^m(B)$. Conversely, if $z\in g^m(B)$, then $f^m(z)\in B$, so the forward orbit of $z$ converges to $0$. Since $f(B)\subset B$, applying $g=f^{-1}$ yields $B\subset g(B)$, and therefore
\[
g^m(B)\subset g^{m+1}(B)\qquad (m\ge 0).
\]
Thus $D$ is an increasing union of connected open sets, hence a domain.

The next step is to trap this domain inside a simple explicit set.

\medskip
\noindent\textbf{Step 2: a thin containing set.}
Fix
\[
D_1:=\{(z_1,z_2)\in\C^2:|z_1|<25,\ |z_2|<5\}
\]
and
\[
D_2:=\{(z_1,z_2)\in\C^2:|z_1|\ge 10+|z_2|\}.
\]
We claim that
\begin{equation}\label{eq:g_invariance}
g(D_1)\subset D_1\cup D_2,
\qquad
g(D_2)\subset D_2.
\end{equation}

To prove the first inclusion, let $w=(w_1,w_2)\in D_1$ and write
\[
z=g(w)=(w_1^2-2w_2,w_1).
\]
If $|w_1|<5$, then $|z_2|=|w_1|<5$. If also $|z_1|<25$, then $z\in D_1$. Otherwise $|z_1|\ge 25$, and since $|z_2|<5$ we obtain
\[
|z_1|\ge 25>10+|z_2|,
\]
so $z\in D_2$.

Now suppose $5\le |w_1|<25$. Then $|z_2|=|w_1|$ and $|w_2|<5$, so
\[
|z_1|=|w_1^2-2w_2|\ge |w_1|^2-2|w_2|>|w_1|^2-10.
\]
Since
\[
|w_1|^2-10-(10+|w_1|)=|w_1|^2-|w_1|-20=(|w_1|-5)(|w_1|+4)\ge 0,
\]
it follows that
\[
|z_1|>10+|z_2|,
\]
so again $z\in D_2$. This proves $g(D_1)\subset D_1\cup D_2$.

For the second inclusion, let $w=(w_1,w_2)\in D_2$ and again put $z=g(w)$. Then
\[
|w_1|\ge 10+|w_2|,
\]
so in particular $|w_1|\ge 10$. Using $|w_2|\le |w_1|-10$, we obtain
\begin{align*}
|z_1|-|z_2|
&=|w_1^2-2w_2|-|w_1| \\
&\ge |w_1|^2-2|w_2|-|w_1| \\
&\ge |w_1|^2-2(|w_1|-10)-|w_1| \\
&=|w_1|^2-3|w_1|+20.
\end{align*}
Since $|w_1|\ge 10$, the last quantity is at least $90$, and therefore
\[
|z_1|-|z_2|\ge 90>10.
\]
Hence $|z_1|\ge 10+|z_2|$, so $z\in D_2$. This proves \eqref{eq:g_invariance}.

Because $B\subset D_1$, an induction based on \eqref{eq:g_invariance} yields
\[
g^m(B)\subset D_1\cup D_2
\qquad\text{for every }m\ge 0.
\]
Combining this with \eqref{eq:D_union}, we obtain
\begin{equation}\label{eq:D_thin}
D\subset D_1\cup D_2.
\end{equation}
In particular, $D$ is a proper subset of $\C^2$, since for example $(0,10)\notin D_1\cup D_2$.

At this point, Theorem~\ref{thm:RR} implies that the attracting basin is a Fatou--Bieberbach domain.

\medskip
\noindent\textbf{Step 3: $D$ is a Fatou--Bieberbach domain.}
The derivative of $f$ at the origin is
\[
Df(0)=
\begin{pmatrix}
0 & 1 \\
-\frac12 & 0
\end{pmatrix}.
\]
Hence
\[
Df(0)^2=-\frac12 I,
\]
so every eigenvalue $\lambda$ of $Df(0)$ satisfies $\lambda^2=-\frac12$, and therefore
\[
|\lambda|=\frac{1}{\sqrt2}<1.
\]
By Theorem~\ref{thm:RR}, the basin $D$ is biholomorphic to $\C^2$. Since $D$ is proper by \eqref{eq:D_thin}, it is a Fatou--Bieberbach domain.

It remains to place this domain inside the desired parabolic tube.

\medskip
\noindent\textbf{Step 4: straightening into $S$.}
Define a linear automorphism of $\C^2$ by
\[
L(z_1,z_2):=\left(\frac{z_2}{25},\frac{z_1}{25}\right)=:(x,y),
\]
and a polynomial automorphism by
\[
A(x,y):=(x+y^2,y)=:(u,v).
\]
We claim that
\[
L(D)\subset T:=\{(x,y)\in\C^2:|x|<1+|y|\}.
\]
Indeed, if $(z_1,z_2)\in D_1$, then
\[
|x|=\frac{|z_2|}{25}<\frac15<1\le 1+|y|.
\]
If $(z_1,z_2)\in D_2$, then $|z_2|\le |z_1|-10$, so
\[
|x|=\frac{|z_2|}{25}\le \frac{|z_1|-10}{25}=|y|-\frac25<1+|y|.
\]
Together with \eqref{eq:D_thin}, this proves $L(D)\subset T$.

Now set
\[
\Omega:=A(L(D)).
\]
Since $A$ and $L$ are automorphisms of $\C^2$ and $D$ is a Fatou--Bieberbach domain, so is $\Omega$. Moreover, if $(u,v)=A(x,y)$ with $(x,y)\in T$, then
\[
u-v^2=x,
\qquad v=y,
\]
whence
\[
|u-v^2|=|x|<1+|y|=1+|v|.
\]
Thus $\Omega\subset S$, as required.
\end{proof}

\begin{remark}
The present construction is in the same general spirit as thin Fatou--Bieberbach domain constructions; compare \cite{BH00}. In the argument above, however, all dynamical estimates needed for Proposition~\ref{prop:fb-parabola} are verified directly, and the only external input is Theorem~\ref{thm:RR}.
\end{remark}

\section{Normality of the associated affine family}\label{sec:normal}

For $(u,v)\in S$, define
\begin{equation}\label{eq:def-f}
    f_{u,v}(z):=u+vz.
\end{equation}
Set
\[
\mathcal{A}_S:=\{f_{u,v}:(u,v)\in S\}.
\]

The point of Proposition~\ref{prop:fb-parabola} is to ensure that the coefficient pair 
$(u,v)$ arising from the biholomorphism in Section~\ref{sec:proof} lies in $S$. We now verify that $S$ was chosen precisely so that the corresponding affine family is normal.

\begin{lemma}\label{lem:normal-affine}
The family $\mathcal{A}_S$ is normal on $\C$.
\end{lemma}

\begin{proof}
Let
\[
f_n(z)=u_n+v_n z,
\qquad (u_n,v_n)\in S,
\]
be an arbitrary sequence in $\mathcal{A}_S$. By the definition of $S$, there exists $\delta_n\in\C$ such that
\begin{equation}\label{eq:delta}
u_n=v_n^2+\delta_n,
\qquad |\delta_n|<1+|v_n|.
\end{equation}
Passing to a subsequence if necessary, we may assume either that $\{v_n\}$ is bounded or that $|v_n|\to\infty$.

\smallskip
\noindent\textbf{Case 1: $\{v_n\}$ is bounded.}
Passing to a further subsequence, we may assume that $v_n\to v\in\C$. By \eqref{eq:delta}, the sequence $\{\delta_n\}$ is then bounded as well, so after another subsequence we may also assume that $\delta_n\to\delta\in\C$. Hence
\[
u_n=v_n^2+\delta_n\longrightarrow v^2+\delta=:u,
\]
and therefore
\[
f_n(z)=u_n+v_n z\longrightarrow u+vz
\]
locally uniformly on $\C$.

\smallskip
\noindent\textbf{Case 2: $|v_n|\to\infty$.}
Let $K\subset\C$ be compact, and write
\[
R_K:=\sup_{z\in K}|z|.
\]
Using \eqref{eq:delta}, for every $z\in K$ we obtain
\begin{align*}
|f_n(z)|
&=|v_n^2+zv_n+\delta_n| \\
&\ge |v_n|^2-|z|\,|v_n|-|\delta_n| \\
&\ge |v_n|^2-R_K|v_n|-(1+|v_n|).
\end{align*}
Therefore
\[
\inf_{z\in K}|f_n(z)|\ge |v_n|^2-(R_K+1)|v_n|-1\longrightarrow\infty.
\]
Equivalently, $f_n\to\infty$ uniformly on $K$ in the spherical metric.

In both cases we have found a subsequence converging uniformly on compact subsets of $\C$ in the spherical metric. Hence $\mathcal{A}_S$ is normal on $\C$.
\end{proof}

\section{Proof of the main theorem}\label{sec:proof}

We are now ready to combine the geometric input from Section~\ref{sec:fb} with the normality statement from Section~\ref{sec:normal}.

\begin{proof}[Proof of Theorem~\ref{thm:main}]
By Proposition~\ref{prop:fb-parabola}, there exists a Fatou--Bieberbach domain $\Omega\subset S$. Choose a biholomorphism
\[
\Phi=(U,V):\C^2_{(a,b)}\longrightarrow \Omega\subset\C^2_{(u,v)}.
\]
Since $\Phi$ is holomorphic on all of $\C^2$, its components $U$ and $V$ are entire functions of $(a,b)$. Define
\begin{equation}\label{eq:defF}
F(z,a,b):=U(a,b)+zV(a,b).
\end{equation}
Then $F\in\OO(\C^3)$.

For each $(a,b)\in\C^2$, we have $(U(a,b),V(a,b))\in\Omega\subset S$, so
\[
F(\,\cdot\,,a,b)=f_{U(a,b),V(a,b)}\in\mathcal{A}_S,
\]
where $f_{U(a,b),V(a,b)}$ is defined by \eqref{eq:def-f}.
Hence
\[
\{F(\,\cdot\,,a,b):a,b\in\C\}\subset \mathcal{A}_S.
\]
By Lemma~\ref{lem:normal-affine}, the family $\mathcal{A}_S$ is normal on $\C$, and every subfamily of a normal family is normal. Therefore
\[
\{F(\,\cdot\,,a,b):a,b\in\C\}
\]
is normal on $\C$.

It remains to prove the non-factorization statement. Assume to the contrary that
\[
F(z,a,b)=G\bigl(z,H(a,b)\bigr)
\]
for some entire functions $G,H\in\OO(\C^2)$. Write the second variable of $G$ as $w$. By the chain rule,
\[
F_a(z,a,b)=G_w\bigl(z,H(a,b)\bigr)\,H_a(a,b),
\qquad
F_b(z,a,b)=G_w\bigl(z,H(a,b)\bigr)\,H_b(a,b),
\]
and
\[
F_{a,z}(z,a,b)=G_{zw}\bigl(z,H(a,b)\bigr)\,H_a(a,b),
\qquad
F_{b,z}(z,a,b)=G_{zw}\bigl(z,H(a,b)\bigr)\,H_b(a,b).
\]
Therefore
\[
F_bF_{a,z}-F_aF_{b,z}\equiv 0
\qquad\text{on }\C^3.
\]

For the function defined by \eqref{eq:defF}, however, a direct computation gives
\[
F_a=U_a+zV_a,
\qquad
F_b=U_b+zV_b,
\qquad
F_{a,z}=V_a,
\qquad
F_{b,z}=V_b,
\]
so
\[
F_bF_{a,z}-F_aF_{b,z}
=
(U_b+zV_b)V_a-(U_a+zV_a)V_b
=
U_bV_a-U_aV_b
=
-\det D\Phi.
\]
Since $\Phi=(U,V):\C^2\to\Omega$ is biholomorphic and $\Phi^{-1}$ is holomorphic, the holomorphic inverse function theorem implies that
\[
\det D\Phi\neq 0
\qquad\text{everywhere on }\C^2.
\]
Hence
\[
F_bF_{a,z}-F_aF_{b,z}\neq 0
\qquad\text{everywhere on }\C^3,
\]
which contradicts the factorization assumption. This proves part~\textup{(ii)}, and thus completes the proof of Theorem~\ref{thm:main}.
\end{proof}

\begin{remark}
The differential criterion mentioned in the original problem is therefore built into the present construction. Indeed,
\[
F_bF_{a,z}-F_aF_{b,z}=-\det D\Phi,
\]
so the obstruction to factorization is exactly the nonvanishing Jacobian of the biholomorphism $\Phi$.
\end{remark}

	\section*{Declaration of competing interest}
	The authors declare no competing interests.
	
	\section*{Data availability}
	No data was used for the research described in the article.
	
	\section*{Acknowledgments}
Teng Zhang is supported by the China Scholarship Council, the Young Elite Scientists Sponsorship Program for PhD Students (China Association for Science and Technology), and the Fundamental Research Funds for the Central Universities at Xi'an Jiaotong University (Grant No.~xzy022024045).


\begin{thebibliography}{99}

\bibitem{Ber06}
W.~Bergweiler, \emph{Bloch's principle}, Comput. Methods Funct. Theory
\textbf{6} (2006), no.~1, 77--108.
\href{https://doi.org/10.1007/BF03321119}{doi:10.1007/BF03321119}.

\bibitem{BH89}
D.~A.~Brannan and W.~K.~Hayman, \emph{Research problems in complex analysis},
Bull. London Math. Soc. \textbf{21} (1989), no.~1, 1--35.
\href{https://doi.org/10.1112/blms/21.1.1}{doi:10.1112/blms/21.1.1}.

\bibitem{BH00}
G.~T.~Buzzard and J.~H.~Hubbard, \emph{A Fatou--Bieberbach domain avoiding a neighborhood of a variety of codimension~2},
Math. Ann. \textbf{316} (2000), no.~4, 699--702.
\href{https://doi.org/10.1007/s002080050350}{doi:10.1007/s002080050350}.

\bibitem{CFZ04}
J.~M.~Chang, M.~L.~Fang, and L.~Zalcman, \emph{Normal families of holomorphic functions},
Illinois J. Math. \textbf{48} (2004), no.~1, 319--337.
\href{https://doi.org/10.1215/ijm/1258136186}{doi:10.1215/ijm/1258136186}.

\bibitem{CH65}
J.~Clunie and W.~K.~Hayman, \emph{The spherical derivative of integral and meromorphic functions},
Comment. Math. Helv. \textbf{40} (1965), 117--148.
\href{https://doi.org/10.1007/BF02564366}{doi:10.1007/BF02564366}.

\bibitem{Ere03}
A.~Eremenko, \emph{Progress in entire and meromorphic functions},
report on problems from Hayman's lists, May 15, 2003.
Available at
\href{https://www.math.purdue.edu/~eremenko/dvi/progr.pdf}{https://www.math.purdue.edu/~eremenko/dvi/progr.pdf}.

\bibitem{Ere06}
A.~Eremenko, \emph{Exceptional values in holomorphic families of entire functions},
Michigan Math. J. \textbf{54} (2006), no.~3, 687--696.
\href{https://doi.org/10.1307/mmj/1163789921}{doi:10.1307/mmj/1163789921}.

\bibitem{Ere10}
A.~Eremenko, \emph{Brody curves omitting hyperplanes},
Ann. Acad. Sci. Fenn. Math. \textbf{35} (2010), 565--570.
\href{https://doi.org/10.5186/aasfm.2010.3534}{doi:10.5186/aasfm.2010.3534}.

\bibitem{EL92}
A.~Eremenko and M.~Yu.~Lyubich, \emph{Dynamical properties of some classes of entire functions},
Ann. Inst. Fourier (Grenoble) \textbf{42} (1992), no.~4, 989--1020.
\href{https://doi.org/10.5802/aif.1318}{doi:10.5802/aif.1318}.

\bibitem{ER96}
A.~Eremenko and L.~A.~Rubel, \emph{The arithmetic of entire functions under composition},
Adv. Math. \textbf{124} (1996), no.~2, 334--354.
\href{https://doi.org/10.1006/aima.1996.0087}{doi:10.1006/aima.1996.0087}.

\bibitem{For11}
F.~Forstneri\v{c}, \emph{Stein Manifolds and Holomorphic Mappings: The Homotopy Principle in Complex Analysis},
Ergebnisse der Mathematik und ihrer Grenzgebiete. 3. Folge. A Series of Modern Surveys in Mathematics, vol.~56,
Springer, Heidelberg, 2011.


\bibitem{Gro68}
F.~Gross, \emph{On factorization of meromorphic functions},
Trans. Amer. Math. Soc. \textbf{131} (1968), 215--222.
\href{https://doi.org/10.2307/1994691}{doi:10.2307/1994691}.

\bibitem{GY72}
F.~Gross and C.-C.~Yang, \emph{The fix-points and factorization of meromorphic functions},
Trans. Amer. Math. Soc. \textbf{168} (1972), 211--219.
\href{https://doi.org/10.1090/S0002-9947-1972-0301175-2}{doi:10.1090/S0002-9947-1972-0301175-2}.

\bibitem{HL19}
W.~K.~Hayman and E.~F.~Lingham, \emph{Research Problems in Function Theory: Fiftieth Anniversary Edition},
Springer, Cham, 2019.
\href{https://doi.org/10.1007/978-3-030-25165-9}{doi:10.1007/978-3-030-25165-9}.

\bibitem{Mar31}
F.~Marty, \emph{Recherches sur la r\'epartition des valeurs d'une fonction m\'eromorphe},
Ann. Fac. Sci. Toulouse Sci. Math. Sci. Phys. \textbf{23} (1931), 183--261.
\href{https://doi.org/10.5802/afst.367}{doi:10.5802/afst.367}.

\bibitem{Mon27}
P.~Montel, \emph{Le\c{c}ons sur les familles normales de fonctions analytiques et leurs applications},
Gauthier-Villars, Paris, 1927.

\bibitem{Ng01}
T.~W.~Ng, \emph{Imprimitive parametrization of analytic curves and factorizations of entire functions},
J. London Math. Soc. (2) \textbf{64} (2001), no.~2, 385--394.
\href{https://doi.org/10.1112/S0024610701002472}{doi:10.1112/S0024610701002472}.

\bibitem{PZ00}
X.~Pang and L.~Zalcman, \emph{Normal families and shared values},
Bull. London Math. Soc. \textbf{32} (2000), no.~3, 325--331.
\href{https://doi.org/10.1112/S002460939900644X}{doi:10.1112/S002460939900644X}.

\bibitem{Rit22}
J.~F.~Ritt, \emph{Prime and composite polynomials},
Trans. Amer. Math. Soc. \textbf{23} (1922), no.~1, 51--66.

\bibitem{RR88}
J.-P.~Rosay and W.~Rudin, \emph{Holomorphic maps from $\C^n$ to $\C^n$},
Trans. Amer. Math. Soc. \textbf{310} (1988), no.~1, 47--86.
\href{https://doi.org/10.1090/S0002-9947-1988-0929658-4}{doi:10.1090/S0002-9947-1988-0929658-4}.

\bibitem{RY95}
L.~A.~Rubel and C.-C.~Yang, \emph{The factorization of $A(z)+B(w)$ under composition},
Illinois J. Math. \textbf{39} (1995), no.~2, 258--270.
\href{https://doi.org/10.1215/ijm/1255986547}{doi:10.1215/ijm/1255986547}.

\bibitem{Zal98}
L.~Zalcman, \emph{Normal families: new perspectives},
Bull. Amer. Math. Soc. (N.S.) \textbf{35} (1998), no.~3, 215--230.
\href{https://doi.org/10.1090/S0273-0979-98-00755-1}{doi:10.1090/S0273-0979-98-00755-1}.

\end{thebibliography}
\end{document}